\documentclass[12pt]{amsart}

 \usepackage{palatino}
 \usepackage{amssymb}
\usepackage[english]{babel}
\usepackage[latin1]{inputenc}
\usepackage{amsmath,amssymb,amsthm, fancyhdr,latexsym}
\usepackage{eurosym}
 \usepackage{avant}
\usepackage{graphicx}
\usepackage{subfigure}
\usepackage{color}
\newtheorem{theorem}{Theorem}[section]
\newtheorem{corollary}[theorem]{Corollary}
\newtheorem{proposition}[theorem]{Proposition}
\newtheorem{lemma}[theorem]{Lemma}

\newtheorem{class}[theorem]{}

\theoremstyle{definition}
\newtheorem{definition}[theorem]{Definition}
\newtheorem{example}[theorem]{Example}
\newtheorem{remark}[theorem]{Remark}

     \newenvironment{prob}{  \noindent  {\bf Problem:}  }{ \vspace{0.3cm}}

\makeindex

\newcommand{\NN}{\mathbb{N}}
\newcommand{\ZZ}{\mathbb{Z}}

\newcommand{\Mf}{{\mathcal{M}\!f}}
\newcommand{\cG}{{\sf{G}}}
\newcommand{\cN}  {{\sf{N}}}
\newcommand{\cT}{{\sf{T}}}
  \newcommand{\Ht}{\textrm{Ht}}
  \newcommand{\supp}{\textrm{Supp}}

\newcommand{\off}{\textrm{off}}

\newcommand{\mult}{\textrm{mult}}
\newcommand{\Gr}{Gr\"obner}

\newcommand{\MFFunctor}[1]{\underline{\mathbf{Mf}}_{#1}}
\newcommand{\MFScheme}[1]{\mathbf{Mf}_{#1}}
\newcommand{\Sets}{\underline{\textnormal{Set}}}

\def\then{\;\Longrightarrow\;}

\def\Bbb#1{{\mathbb #1}}

\def \clss{\mathop{\rm class}\nolimits}

\title{ Term-ordering free involutive bases   }

\author[M.~Ceria]{Michela Ceria}
\address{Michela Ceria\\ Dipartimento di Matematica dell'Universit\`{a} di Torino\\ 
         Via Carlo Alberto 10\\     10123 Torino\\ Italy.}
\email{michela.ceria@unito.it}

\author[T.~Mora]{Teo Mora}
\address{Teo Mora\\ Dipartimento di Matematica dell'Universit\`{a} di
Genova\\
         Via Dodecaneso 35\\     16146 Genova\\ Italy.}
\email{theomora@disi.unige.it}

\author[M.~Roggero]{Margherita Roggero}
\address{Margherita Roggero\\ Dipartimento di Matematica dell'Universit\`{a} di Torino\\ 
         Via Carlo Alberto 10\\     10123 Torino\\ Italy.}
\email{margherita.roggero@unito.it}

\thanks{The third author was supported by the framework of PRIN 2010-11 \emph{Geometria delle variet\`a algebriche}, cofinanced by MIUR}

\begin{document}
\newpage
\fussy
\begin{abstract} 
In this paper, we consider a monomial ideal $J \triangleleft P:=A[x_1,\dots ,x_n]$, over a commutative ring $A$, and we face the problem of the characterization 
for the family $\Mf(J)$ of all homogeneous ideals $I \triangleleft P$ such  that the $A$-module $P/I$ is  free  with basis given by the set of terms in the
 \Gr\  escalier $\cN(J) $ of $J$. This family is in general wider than that of the  ideals having $J$ as initial ideal w.r.t. any term-ordering, hence
 more suited to  a computational approach 
 to the study of Hilbert schemes.\\ 
For this purpose, we exploit and enhance the concepts of multiplicative variables, complete sets and involutive bases introduced by Janet
in \cite{J1,J2,J3} and we generalize the construction of $J$-marked bases and term-ordering free reduction process introduced and deeply studied in 
\cite{BCLR,CR} for the special case of a strongly stable monomial ideal $J$.\\
Here, we introduce and characterize  for every monomial ideal $J$ a particular complete set of generators $\mathcal F(J)$, called stably complete, 
that allows an explicit description of the family $\Mf(J)$. We obtain stronger results if $J$ is quasi stable, proving that  $\mathcal F(J)$
is a  Pommaret basis  and  $\Mf(J)$ has a natural  structure  of affine  scheme. 

 The final section presents a detailed  analysis of the origin and the historical evolution of the main  notions  we refer to.
\end{abstract}
\maketitle
\section{Introduction.}\label{Intro} 
 Let $P:=A[x_1,...,x_n]$  be the polynomial ring in $n$ variables over a commutative ring $A$. 
The problem we address is the following.

\vspace{0.3cm}
\begin{prob}\label{Classif} \textit{Given any   monomial ideal $J\triangleleft P$ find a 
  characterization for  the  family $\Mf(J)$ of all homogeneous ideals $I \triangleleft P$ such  that the $A$-module $P/I$ is  free  with basis given by 
  the set of terms in the
 \Gr\ escalier $\cN(J) $ of $J$. }
\end{prob}

The most relevant examples of  ideals that 
 belong to this family are those
  such that $J$ is their
initial ideal w.r.t. some term-ordering, but
 in general they  form  a  proper
 subset of $\Mf(J)$. Therefore, we must overcome the  \Gr \ framework.  \\ 
  A computational description of  the  whole family $\Mf(J)$ is obtained  in  \cite{BCLR,CR} for  $J$  strongly stable. 
These families are optimal  for many  applications, for instance for an effective study of  Hilbert schemes (see \cite{BLR}). However,
the strong stability of the monomial ideal
   $J$ is a rather limiting condition.

 In the present paper we
give an overall view on what can be said about the above question for an
 \emph{arbitrary} monomial ideal $J$,
enhancing some ideas  introduced  by Janet in \cite{J1,J2,J3}.
\\
The ideas we  mainly deal with are those of \emph{multiplicative variable} and \emph{complete system}, leading to the so called \emph{Janet decomposition} 
for terms.  
These concepts date back to the late   nineteenth century and the first decades of the twentieth.  In a historical note at the end of the paper we present 
a  detailed overview of their  appearances,  evolution and applications.

In Janet's theory the ideals $I$ are generated by those we call now \emph{involutive bases}, a set which contains as a subset  those that are also \Gr\ bases.
Indeed,  Janet develops his ideas
 assuming  to be in \emph{generic coordinates}. Hence  the homogeneous ideals $I$ and $J$ he considers satisfy  many good properties that always hold after having  
 performed a \emph{generic linear change of coordinates}.  In particular, $J$ is the generic initial ideal of $I$ w.r.t. the (deg)-revlex ordering. 

 From a computational
 point of view, a general change of coordinates    is  remarkably heavy.  
For this reason, we consider interesting enhancing the theory
and obtaining analogous results not assuming this hypothesis. 
Indeed, Janet's   ideas permit to go beyond this context and to recover results and  techniques of both  \Gr\  
basis theory and $J$-marked basis theory. In fact 
 we do not need to impose a term-ordering on the given polynomial ring. 

We  identify two essential features that are key points for most computations  in both the above  frameworks:
\begin{itemize}
 \item[I)] $I$   is generated by a set of   polynomials, marked on the terms of a suitable generating set of the monomial ideal $J$;
 \item[II)]there is a reduction process w.r.t. these marked polynomials,  that  is used to rewrite each element of $P/I$ as an element of the free 
 $A$-module $\langle \cN(J) \rangle$
\end{itemize}
Janet's notions of multiplicative variable and complete system allow to construct such marked set of generators for $I$ and to define an
efficient reduction process.

First of all, we examine and compare two different definitions of multiplicative variable that Janet presents  in \cite{J1,J2} and in \cite{J3}, that are
equivalent in general coordinates. We  underline  similarities and differences and  introduce the notion of \emph{stably complete} set of terms, when both
conditions hold. We show that every monomial ideal $J$ has only one stably complete set of generators (possibly made of infinitely many terms) that we call 
\emph{star set} and  denote by $\mathcal{F}(J)$.

Furthermore, we define a reduction procedure with respect to a homogeneous  set of polynomials  marked on  a stably complete system $\mathcal{F}(J)$ and prove its 
noetherianity.  As a consequence we are able to give a first, general answer to Problem \ref{Classif} .
  
Of course, the most interesting
 cases are those of ideals $J$ such that their generating stably complete set  $M$ is finite. We prove that   they are  the \emph{ quasi stable}
 ideals and that $\mathcal F(J)$ is their Pommaret basis. Among them,    those   such that $\mathcal{F}(J)$ coincides with the monomial basis  are exactly the 
 \emph{ stable} ones.\\
For the class of quasi stable ideals $J$   we   give a more complete and effective answer to  Problem \ref{Classif}.
Indeed, we prove that our description of $\Mf(J)$ is natural, in the sense that  it defines a    representable  functor from the category of $\ZZ$-algebras to the category of sets. Finally we   give an effective  procedure computing  equations for the  scheme that represents this functor.

After introducing all the notation (section \ref{Notat}),
 we introduce the Janet decomposition for semigroup
ideals and order ideals (section \ref{Jdec}).
\\
More precisely, we recall the notion
of multiplicative variables (\cite{J1,J2,J3})
 and complete system, pointing out that in the cited
 papers Janet uses two \emph{non equivalent} definitions
of multiplicative variables and
completeness.

For our purpose, we then introduce the notion
of \emph{stable completeness} as a
junction between the two notions.
 Given a complete system of terms $M$, we define a
decomposition for terms in $J=(M)$, called
\emph{star decomposition}, in analogy with
the star product introduced in \cite{BCLR,CR}
for strongly stable ideals.
\\
 In section
 \ref{StarSet} we define a very special stably complete set,
 i.e. the \emph{star set} $\mathcal{F}(J)$, introducing the
 \emph{stable ideals}, for which $\mathcal{F}(J)$ is the minimal
generating set of $J$  and the \emph{quasi stable ideals} for which 
the finiteness condition is respected.
\\
In section \ref{sec:reductiont},  we define
$M$-marked polynomials, bases and families,
for a complete set $M$, again generalizing
the definitions given for the generating set of a strongly stable ideal. Then, we define a noetherian
reduction process for homogeneous polynomial w.r.t.
the elements of a $M$-marked set.
In  section \ref{MFam} we associate to marked families a representable functor and give  a procedure computing equations for the scheme that defines it.

Finally, section \ref{Hist} contains the historical note.
\section{Notation.}\label{Notat}
Consider the polynomial ring
$P:=A[x_1,...,x_n]=\bigoplus_{d \in \NN}P_d$
in $n$ variables  and coefficients in the base
ring $A$.\\
When an order on the variables comes into play,
we consider $x_1<x_2<...<x_n$.\\
In case of $n=2,3$ we will usually set $P=A[x,y]$, $x<y$ and $P=A[x,y,z]$, $x<y<z$.\\
The symbol $<_{Lex}$ will denote the lexicographic
term-ordering according to this order
on the set  of variables.\\
The \emph{set of terms} of $P$ is
$$\mathcal{T}:=\{x^{\alpha}=x_1^{\alpha_1}\cdots x_n^{\alpha_n},\,(\alpha_1,...,\alpha_n)\in \NN^n \}$$
and we define also
$$\mathcal{T}[1,m]:=\mathcal{T}\cap k[x_1,...,x_m]=\{x_1^{\alpha_1}\cdots x_m^{\alpha_m}/\, (\alpha_1,...,\alpha_m)\in \NN^m \}.$$
For every polynomial $f\in P$,
$deg(f)$ is its usual degree and $deg_i(f)$ is
 its degree with respect to the variable $x_i$.\\
For each $p \in \NN$, and for all $W \subseteq P$,
$$W_p:=\{f \in W \ :  \ f  \,\textrm{homogeneous and }  \, deg(f)=p\};$$
in particular:
$$\mathcal{T}_p:=\{\tau \in \mathcal{T}\, : \, deg(\tau)=p\},\;\; \vert \mathcal{T}_p \vert =dim_k(P_p)={p+n-1\choose n-1}.$$
If $f \in P$, we
 denote
by $\supp (f)$
the \emph{support} of $f$, i.e. the set of all the terms in $\mathcal{T}$, appearing in $f$ with non-zero coefficient.\\

Fixed a polynomial $f \in \mathcal{P}$ and a term-ordering $<$, we call \emph{leading term} of $f$ the maximal element in $\supp(f)$
w.r.t. $<$ and we denote it $\cT(f)$. Its coefficient is the \emph{leading coefficient} of $f$.
Given a term $\tau=x_1^{\alpha_1}\cdots x_n^{\alpha_n} \in \mathcal{T}$,  we set
$$\max(\tau)=\max\{x_i\ :\  x_i\mid \tau \}   \textit \quad , \quad  \min(\tau)=\min\{x_i\  :\  x_i \mid \tau \}$$
the maximal and the minimal variable appearing
 in $\tau$ with non-zero exponent.
\begin{definition}\label{Predecessor}
Given a term $\tau \in \mathcal{T}$ and
a variable $x_j\mid \tau$, the term $\frac{\tau}{x_j}$ is the $j$-th \emph{predecessor} of $\tau$.
\end{definition}

\begin{definition}\label{Syzygy}
Let $F=\{  \tau_1,...,\tau_s\} \subseteq \mathcal{T}$
be an ordered subset of terms, generating an ideal $J=(F)$. The module
$$Syz(F)=\{(g_1,...,g_s)\in P^s,\, \sum_{i=1}^s g_i \tau_i=0 \} $$
 is the \emph{syzygy module } of $F$.\\
We denote an element in $Syz(F)$  as $(g_1,...,g_s)$  and we  call it \emph{syzygy among } $F$.
\end{definition}

\begin{definition}
A set $\cN   \subset \mathcal{T}$ is called
 \emph{order ideal} if $$ \forall s  \in \mathcal{T},
\ t \in \cN  : \ \  s\vert t \Rightarrow s \in \cN  .$$
\end{definition}
Observe that $\cN  $ is an order ideal if
and only if the complementary set
$I:=\mathcal{T}\setminus \cN  $ is a \emph{semigroup ideal},
i.e. $\forall t \in \mathcal{T},\, \tau \in I \Rightarrow t\tau \in I $.\\
 If $I$ is either a monomial ideal
 or a semigroup ideal, we will denote
by $\cN  (I)$ the order ideal $\cN  :=\mathcal{T}\setminus I$
and by
 $ \cG  (I)$ its \emph{ monomial basis}, namely the minimal
set of terms generating $I$.
\begin{definition}[\cite{RS}]
A \emph{(monic) marked polynomial} is a polynomial
 $f\in P$ together with a fixed term $\tau$ of $\supp (f)$,
 called \emph{head term of} $f$ and denoted by
$\Ht(f)$ and such that its coefficient is equal to $1_A$.
\end{definition}
We can extend to  marked polynomials the notion of
$S$-polynomial.

\section{Janet decomposition.}\label{Jdec}
In this section we loosely base on the paper  \cite{J1},
where  Janet
 first defines the notion of \emph{multiplicative variable}
 for a term $\tau$ with respect to a given  set $M\subseteq \mathcal{T}$.

For completeness' sake, we recall Janet's decomposition for  terms in the semigroup ideal generated by $M$ into disjoint classes.
 
 Each of them contains:
\begin{enumerate}
 \item  a term $\tau \in M$;
 \item the set of monomials obtained  multiplying
$\tau$ by products of multiplicative variables,
that we call  \emph{ offspring of} $\tau$ and denote $\off_M(\tau)$.
\end{enumerate}

The main difference with respect to Janet's papers is
that we remove the finiteness condition on $M$, showing that it is not necessary for our purposes.

\begin{definition}\cite[ppg.75-9]{J1}\label{multiplicative}
Let  $M\subset \mathcal{T}$ be a set of terms
 and  $\tau=x_1^{\alpha_1}\cdots x_n^{\alpha_n} $
be an element of $M$.
A variable $x_j$ is called \emph{multiplicative}
for $\tau$ with respect to $M$ if there is no term in
$M$ of the form
$\tau'=x_1^{\beta_1}\cdots x_j^{\beta_j}x_{j+1}^{\alpha_{j+1}} \cdots  x_n^{\alpha_n}$
with $\beta_j>\alpha_j$.
 We will denote by $\mult_M(\tau)$ the set of
multiplicative variables for $\tau$ with respect to $M$.
\end{definition}
\begin{definition}
With the previous notation, the \emph{offspring} of  $\tau $ with respect to $M$   is the set
$$\off_M (\tau):=\{\tau x_1^{\lambda_1} \cdots x_n^{\lambda_n} \,\vert \, \textrm{where } \lambda_j\neq 0 \textrm{ only if } x_j \textrm{ is multiplicative for }
\tau \textrm{ w.r.t. } M\}.$$
\end{definition}
\begin{example}
Consider the set
$M=\{x_1^3,x_2^3, x_1^4x_2x_3,x_3^2\}\subseteq k[x_1,x_2,x_3].$\\
Let $\tau=x_1^3$, so $\alpha_1=3,\,\alpha_2=\alpha_3=0$.
The variable $x_1$ is multiplicative for $\tau$ w.r.t $M$
since there are no terms
 $\tau'=x_1^{\beta_1}x_2^{\beta_2}x_3^{\beta_3}\in M$ satisfying both conditions:
\begin{itemize}
 \item $\beta_1 > 3$;
 \item $\beta_2=\beta_3=0$.
\end{itemize}
On the other hand, $x_2$ is not multiplicative for $\tau$ since $\tau''=x_1^{\gamma_1}x_2^{\gamma_2}x_3^{\gamma_3}=x_2^3\in M$: \\$\gamma_2=3>0=\alpha_2$,
$\gamma_3=\alpha_3=0$.\\
Similarly, $x_3$ is not multiplicative since $x_3^2 \in M$.\\
In conclusion, we have $\mult_M(\tau)=\{x_1\}$.
\end{example}

\begin{remark}\label{oss: solo tau}  Observe that, by definition of multiplicative variable, the only element in $\off_M (\tau)\cap M$ is $\tau$ itself.
\\
Indeed, if $\tau\in M$ and also $\tau \sigma \in  M $ for
a non constant term  $\sigma$, then $max(\sigma)$ cannot be
multiplicative for $\tau$, hence
$\tau \sigma \notin \off_M(\tau)$. \end{remark}

In paper \cite{J1}, Janet defines multiplicative
variables as in Definition \ref{multiplicative} and he
provides both a decomposition for the semigroup ideal
$\cT  (M)$ generated by a finite set of terms $M$
and a decomposition for the complementary set $\cN  (M)$.
\\
On the other hand, in \cite{J2,J3}, he defines
 multiplicative variables in the following way.

\begin{class}\label{Molt2}
A variable $x_j$ is \emph{multiplicative} for
$\tau \in \mathcal{T}$ if and only if $x_j \leq \min(\tau)$.
\end{class}

These two definitions of multiplicative
variables appear to be very different.\\
First of all, in the first formulation,
the set of multiplicative
variables for a term in $M$ depends on the whole set $M$, while in the second it is completely
independent on the set $M$.
Indeed, the two notions are not equivalent
for a general set $M$, as shown by
the following examples.

\begin{example}
In $k[x_1,x_2,x_3]$ consider the ideal $I=(x_1^2x_2,x_1x_2^2)$ and let
 $M$ be its monomial basis.
Then, $\mult_M(x_1^2x_2)=\{x_1,x_3\}$ and $\mult_M(x_1x_2^2)=\{x_1,x_2,x_3\}$
while only $x_1$ can be multiplicative
according to the other notion of multiplicative variable.
\end{example}
\begin{example}\label{Es3.6}
Taken the set $M=\{x_1^2x_2,x_1x_2^2\}\subseteq k[x_1,x_2]$, we get
$\mult_M(x_1x_2^2)=\{x_1,x_2\} ,$
while of course  $x_1\leq \min(x_1x_2^2)$ but $x_2 >\min(x_1x_2^2)$.
\end{example}

However, they are equivalent in Janet setting, that is if $M$ is the generating set of the generic initial ideal of homogeneous ideals $I$.\\
More generally, we will see that they
turn out to be equivalent
also if  $M$ is  the monomial
basis  $\cG  (J)$ of a
strongly stable ideal $J$ and if $M$ is
 the special set of generators
 of  any monomial ideal $J$ that
we will introduce in section
\ref{StarSet} and denote
by $\mathcal F(J)$.

We will see that   stronger results
can be proved  when a set $M$ is such that
\emph{the two definitions} of multiplicative variables
\emph{coincide}.

The following definition will be a key point  in this paper.

\begin{definition}\cite[ppg.75-9]{J1}\label{Complete}
A set of terms $M\subset \mathcal{T}$ is called
\emph{complete} if for every $ \tau \in M$ and
 $ x_j\notin mult_M(\tau)$, there exists $ \tau' \in M$
such that $x_j \tau \in \off_M (\tau')$.

 Moreover,  $M$ is \emph{stably complete}  if it
 is complete and for every $\tau\in M$ it holds
  $\mult_M(\tau)=\{ x_i \  \vert \  x_i\leq \min(\tau) \}$.
\\
If a set $M$ is stably complete and finite, then it is the \emph{Pommaret basis} of $J=(M)$ and we denote it by $\mathcal{H}(J)$.
\end{definition}
\begin{remark}\label{singleton}
If $M=\{\tau\}\subseteq \mathcal{P}$ is a singleton, it is complete, with $\mult(\tau)=\{x_1,...,x_n\}$.
\end{remark}

Let us examine some examples.
\begin{example}\label{completex1}
In $k[x_1,x_2,x_3]$ consider the ideal $I=(x_1^2,x_1x_2,x_3)$.\\
Both $M_0 =\{x_1^2,x_1x_2,x_3\}$ and each generating set  of $I$
with the shape \\
$M_i =\{x_1^2,x_1x_2,x_3,x_2x_3,...,x_2^ix_3\}$ are complete
 systems of terms.
In fact, for $M_0$:
\begin{itemize}
 \item[$\mathbf{-}$] $\mult_{M_0}(x_1^2)=\{x_1\}$,
$x_1^2x_2 \in \off_{M_0}(x_1x_2)$, $x_1^2x_3 \in \off_{M_0}(x_3)$;
 \item[$\mathbf{-}$] $\mult_{M_0}(x_1x_2)=\{x_1,x_2\}$,
 $x_1x_2x_3 \in \off_{M_0}(x_3)$;
 \item[$\mathbf{-}$] $\mult_{M_0}(x_3)=\{x_1,x_2,x_3\}$.
\end{itemize}

For $M_i$, $i\geq 1$:
\begin{itemize}
 \item[$\mathbf{-}$] $\mult_{M_i}(x_1^2)=\{x_1\}$, $x_1^2x_2 \in \off_{M_i}(x_1x_2)$, $x_1^2x_3 \in \off_{M_i}(x_3)$;
 \item[$\mathbf{-}$] $\mult_{M_i}(x_1x_2)=\{x_1,x_2\}$,
 $x_1x_2x_3\in \off_{M_i}(x_2x_3)$;
 \item[$\mathbf{-}$] $\mult_{M_i}(x_3)=\{x_1,x_3\}$,
$x_2x_3 \in \off_{M_i}(x_2x_3)$;
\item[$\mathbf{-}$] $\mult_{M_i}(x_2^jx_3)=\{x_1,x_3\}$,
 $x_2^{j+1}x_3 \in \off_{M_i}(x_2^{j+1}x_3)$, $0\leq j<i$;
 \item[$\mathbf{-}$] $\mult_{M_i}(x_2^ix_3)=\{x_1,x_2,x_3\}$.
\end{itemize}
\end{example}

\begin{example}\label{ComplInf}
Consider the ideal $J=(xy) \triangleleft k[x,y]$.\\
The monomial basis $M_0=\cG  (J)=\{xy\}$
 is a complete system with $\mult_{M_0}(xy)=\{x,y\}$.
\\
Also the set 
$M=\{x^hy\ \vert \ h \geq 1\}\subseteq k[x,y]$, $x<y$,
is a complete system, again according to the first definition. It generates the same ideal $(xy)$, but has infinitely many elements.
Anyway, it is not stably complete.
In fact, for each $x^hy \in M$, $\mult_M(x^hy)=\{y\}$,
since no terms
of the form $x^ly^e$ with $e>1$ belong to $M$;
on the other hand  $x\notin \mult_M(x^hy)$ since
$x^{h+1}y\in M$.
\end{example}

\begin{example}\label{nonbasta}
Let $M$ be the set of terms $\{x,y^2\}$ in
 $ k[x,y]$, with $x<y$.
\\
The multiplicative variables for every term in $M$
 are those lower than or equal to its minimal one:
\\
$\mult(x)=\{x\}$ $\mult(y^2)=\{x,y\}$.
\\
However, $M$ is not complete since $yx$ does
not belong to the offspring of any term in $M$.
\end{example}
The following example shows that a complete generating set of terms
can loose completeness when the ideal is enlarged.
\begin{example}\label{completex2}
Let $M=\{x^2,xy\}\subset k[x,y]$ and $J=(M)$.
It is a complete system, but it is not stably complete, since $y$ is multiplicative for $xy$, although $\min(xy)=x$. 
\\
Adding to $M$ a term in $\cN(J)$, we get a new set $M_0$ and $J_0=(M_0)$, whose Janet decomposition clearly changes.
For example, if $M_0=\{x^2,xy,y^2\}$ we get a stably complete system.\\
On the other hand, if $M_0=\{x^2,xy,y^3\} $ the
system is not complete anymore, since $xy^2$ does not
belong to the offspring of any term in the set.
\end{example}

The following  technical lemma will be very useful
 throughout the paper. As a first application, we will
 prove that a system of terms $M$ (possibly infinite)
 is \emph{complete} if and only
if the offsprings   of the elements in $M$ form a
partition of the semigroup ideal generated by $M$.
\begin{lemma}\cite[pg.23]{J4}\label{DivMagLex}
Let $\tau$, $\tau'$  be  elements of a  set of terms
$M$ and $x_j$ be a variable such that
$x_j \notin \mult_M(\tau)$ and $x_j\tau \in \off_M (\tau')$.
 Then
$\tau<_{Lex}\tau'$.  If, moreover,   $x_j\leq min(\tau)$,
then  $\tau x_j=\tau' \in M$.
\end{lemma}
\begin{proof}  First of all,  we observe that
 $\tau \neq \tau'$, since $x_j\notin mult_M(\tau)$.
By definition of offspring,
we have that $\tau x_j=\tau' \sigma'$,
where $\sigma'$ is a product of
multiplicative variables for $\tau'$.
Let us assume by contradiction that $\tau>_{Lex}\tau'$
and let $x_i$ be the maximal variable such
 that $deg_i(\tau)>deg_i(\tau')$. Then, $x_i\vert \sigma'$, hence $x_i\in mult_M(\tau')$,
but this is impossible by definition of
 multiplicative variable, since also $\tau$ is
in $M$.

Now let us assume that  $x_j\leq min(\tau)$
and  $\sigma' \neq 1$. If   $x_j\vert \sigma'$,
then $\tau =\frac{\sigma'}{x_j}\tau'\in M\cap \off_M(\tau')$,
which is not possible by Remark \ref{oss: solo tau}.
If, on the contrary,  $x_j\not\vert \sigma'$ we get
 a contradiction with the previous assertion,
since in this case $\tau' \leq_{Lex}\frac{\tau'\sigma'}{\max(\sigma')}<_{Lex}\frac{\tau'\sigma'}{x_j}=\tau$.
\end{proof}

\begin{theorem}\label{DISG}
Let  $M$ be a  set of terms (possibly infinite).

If   $ \tau, \tau' \in M$ and  $\tau \neq \tau'$,
then $\off_M (\tau)\cap \off_M (\tau')=\emptyset.$

If, moreover, $M$ is complete and   $\cT(M)$ is the semigroup
 ideal it generates,
then $\forall \gamma \in \cT(M)$, $\exists \tau \in M$
such that $\gamma \in \off_M(\tau)$.
Hence, the offsprings of the elements in
$M$ give a partition of $\cT(M)$.
\end{theorem}
\begin{proof}
To prove the first assertion, let us assume
by contradiction that
$\tau\sigma=\tau'\sigma'  \in \off_M (\tau)\cap \off_M (\tau')\neq\emptyset$
and let
$\tau >_{lex} \tau' $.
If $x_i$ is the maximal variable such that
 $deg_i(\tau)>deg_i(\tau')$, then $x_i\vert \sigma'$.
By definition of offspring,  $x_i\in mult_M(\tau')$,
but this is impossible by definition of multiplicative
 variable, since also $\tau$ is
in $M$.
\medskip

Now we assume that $M$ is complete and
prove the second fact. 
We argue by contradiction.
Suppose $\cT(M)\supsetneq O:=\bigcup_{\sigma \in M} \off_M (\sigma)$
and take any term $\gamma $ in $\cT(M)\setminus O$.
 As $M$ generates $\cT(M)$,  there are terms
  in $M$ that divide $\gamma$: let $\tau$
 be the one  which is  maximal  with respect to  $<_{lex}$.
 If $\gamma=\tau \sigma$,
the term $\sigma$ contains at least a variable
$x_i$ which is not multiplicative for $\tau$,
 since $\tau \sigma\notin \off_M(\tau)$.
Then $\gamma=\tau x_i \eta$ and $\tau x_i \notin \off_M(\tau)$.

  By the
 completeness of  $M$, we have  $\tau x_i \in O$,
namely there is a term $\tau' \in M$ such that $\tau x_i=\tau' \sigma' \in \off_M(\tau')$.
 By Lemma \ref{DivMagLex} \emph{i)}, $\tau'>_{Lex} \tau$,
 and this is not possible since
$\tau'\vert\gamma=\tau x_i \eta=\tau'\sigma'\eta$.
\end{proof}

Thanks to the previous  result,
if $M$ is a   complete system,
each term in $\cT(M)$ can be
written in a unique way  as a product of
\begin{enumerate}
 \item an element $\tau \in M$;
 \item a term $x^{\eta}=x_i^{\eta_i}\cdots x_j^{\eta_j}$, with $x_i,...,x_j \in \mult_M(\tau).$
\end{enumerate}
This fact suggests the following
\begin{definition}\label{StarProd}
Let $M$ be a complete system of terms.
The \emph{star decomposition} of every term
 $\gamma\in (M)$ with respect to $M$,   is the unique   couple of terms $(\tau,\eta)$, with $\tau \in M$,
such that $\gamma=\tau \eta$  and  $\gamma \in \off_M(\tau)$.  If $(\tau,\eta)$ is the star decomposition of $\gamma$ with respect to $M$, we will write
$\gamma=\tau \ast_M \eta$.
\end{definition}

\begin{remark} From the results stated above,
we   obtain the following explicit formula for
the Hilbert function of $P/(M)$: $$H(P/(M))(k)={k+n\choose n}-\sum_{\tau \in M\,deg(\tau)\leq k}{k-deg(\tau)+
s_{\tau}-1  \choose s_{\tau}-1},$$
where $s_{\tau}$ is
the number of multiplicative
variables for $\tau$ w.r.t $M$ and we set equal to
$0$ every  binomial
with a negative  numerator or a negative
denominator.
Thus, this formula makes sense even if
$\vert M\vert =\infty$, since   for every $k$ there are
 only finitely many non-zero summands.\\
If $M$ is a finite set of
terms and $r$ is the maximal
degree of its elements,
this formula gives
 the value of the
Hilbert polynomial for every $k\geq r$.
\end{remark}

The following lemma will be very useful for the
reduction process we will define in section \ref{sec:reductiont}.
\begin{lemma}\label{lem: minorelex}
 Let $M$ be a stably complete system of terms and let
 $\gamma $ be a term such that $\gamma=\tau \ast_M \eta$
and also $\gamma=\sigma \eta'$ with $\sigma \notin \cT(M)$.
\\
 Then $\eta' >_{Lex} \eta$.
\end{lemma}
\begin{proof}
 By definition of stable completeness,
 $\min(\tau)\geq \max(\eta)$. If $\eta'<_{Lex} \eta$,
then $\eta'\vert \eta$ and $\tau \vert \sigma$.
 This is not possible since
 $\tau \in \cT(M)$ and $\sigma \notin \cT(M)$.
\end{proof}

\section{Star set and quasi stable ideals}\label{StarSet}

We introduce here a special set of terms.
We will prove that it is a complete system
with many interesting properties in common with
 the minimal monomial
basis of strongly stable ideals.
\begin{definition}
Given a  monomial ideal $J\triangleleft P$ we define the \emph{star set} as
$$\mathcal{F}(J):=\{x^{\alpha} \in \mathcal{T}\setminus \cN  (J) \, \vert \, \frac{x^{\alpha}}{\min(x^{\alpha})} \in \cN  (J) \}.$$
\end{definition}
\begin{theorem}\label{moltiplicative} For every monomial ideal $J$, the star set $\mathcal{F}(J)$ is the unique stably complete system    
of generators of $J$. Hence, if $M$ is stably complete, $M=\mathcal{F}((M))$.
\end{theorem}
\begin{proof}
Let $\tau:= x_k^{\alpha_k}\cdots x_n^{\alpha_n}$ be any monomial in
$\mathcal{F}(J)$.

Assume $x_i$ is not multiplicative, so that $x_i\tau\in J$,
$x_i\tau=\tau'\sigma',\tau'\in M$. Then Lemma \ref{DivMagLex} implies
$\tau<_{Lex}\tau'$ whence $x_i>\min(\tau)$.

Let $x_i> x_k:=min(\tau)$ and set $\sigma_0:=\tau x_i$,
$\sigma_r:=\frac{\sigma_{r-1}}{min(\sigma_{r-1})}$ for $r=1\dots,
\alpha_k+\dots +\alpha_{i-1}$ and note that $x_i^{\alpha_i}\cdots
x_n^{\alpha_n}\notin J$,
since it divides
$\frac{\tau}{min(\tau)}$, while $\sigma:=\sigma_0\in J$, since it is a multiple
of $\tau$.
Then, in the sequence of terms $\sigma_i, 0\leq i\leq \alpha_k+\dots
+\alpha_{i-1}$, we
find an element $\sigma_j$ that belongs to $J$,
while the following one does not.

Then $\sigma_j\in  \mathcal{F}(J)$,
so that $x_i\tau\in\off_{\mathcal{F}(I)} (\sigma_j)$ and
$x_i$ is not multiplicative for
$\tau$ w.r.t. $\mathcal{F}(I)$.

\medskip

Take $\tau=x_k^{\alpha_k}\cdots x_n^{\alpha_n}\in \mathcal{F}(J) $,
and a  variable $x_i\notin \mult_{\mathcal F(J)}(\tau)$.
By the previous result
$x_i>x_k=\min(\tau)$. By definition of non-multiplicative variable,
there is a term $\sigma'=x_i^{t}\ x_{i+1}^{\alpha_{i+1}}\cdots x_n^{\alpha_n}\in
\mathcal{F}(J)$, for some
integer $t>\alpha_i$. Let us consider the minimum one.

If $t=\alpha_i +1$, then
$x_i \tau =x_k^{\alpha_k}\cdots x_i^{t}\cdots x_n^{\alpha_n}\in
\off_{\mathcal{F}(J)} (\sigma')$.

If, on the contrary, $t>\alpha_i+1$, then  $\sigma''=x_i^{\alpha_i+1}\cdots
x_n^{\alpha_n}\in \cN  (J)$ by definition.
Let us consider, as in the
previous proof, the sequence of terms
$\sigma_0:=\tau x_i\in J$, $\sigma_r:=\frac{\sigma_{r-1}}{min(\sigma_{r-1})}$
for $r=1\dots,\sum_{j=k}^{k-1}\alpha_j$.
Since the  last one is $\sigma''$, we can
find in this  sequence a suitable $\sigma_j\in I$
such that $\sigma_{j+1}\in \cN(J)$, that is $\sigma_j\in \mathcal
F(J)$ and $x_i\tau \in \off_{\mathcal F(J)}(\sigma_j)$.

\medskip

In order to prove that every stably complete set of terms $M$, with $J=(M)$ is
exactly $\mathcal{F}(J)$,
we first notice that clearly $  \cG (J) \subseteq M$ and $  \cG (J) \subseteq
\mathcal{F}(J)$.\\
Moreover, it is sufficient to prove that $\mathcal{F}(J) \subseteq M$. Let
$\sigma \in \mathcal{F}(J)$,
i.e. $\frac{\sigma}{\min(\sigma)}=\omega \in \cN(J).$ Then, there exists $\tau
\in M$ such that $\sigma
\in \off(\tau)$ and so $\sigma = \tau \eta$, with either $\eta =1$ or
$\max(\eta) \leq \min(\tau)$. \\
This implies that either $\tau=\sigma$ or $\tau \mid \omega$, but the second
alternative is impossible since both $\tau \in M$ and $\omega \in \cN(J).$

\end{proof}

\begin{remark}
\begin{enumerate}
\item[i.]
For an arbitrary monomial ideal $J$
the set $\mathcal{F}(J)$ can be infinite.
For example, if
$J=(x)\triangleleft  k[x,y],\, x<y$,
then $\mathcal{F}(J)=\{xy^n  \ \vert \  n \in \NN \}$.
\item[ii.]
Not all the complete systems turn
out to be of the form of a star set.\\
For example,  the complete system
 $M=\{x^hy,\, h \geq 1\}\subseteq k[x,y]$
 of Example \ref{ComplInf} is not the star set of the ideal $J:=(M)$.\\
Indeed,
$\cN  (J)=\{x^m,\, m\geq 0\}\cup \{y^l,\, l>0\}$ and
all the terms of the form $xy^k,$ $k>1$,
do not belong to $M$, even if
$\frac{xy^k}{\min(xy^k)}=y^k\in \cN  (M).$
\\
Moreover, for $h>1$, $\frac{x^hy}{x}=x^{h-1}y\in M$, so $x^hy\notin \mathcal{F}(J)$.
\end{enumerate}
\end{remark}

Better results hold if the monomial ideal
 $J$ satisfies one of the following
 conditions, weaker then the strongly stable property (see section \ref{MFam}).

\begin{definition}\label{Qstab}
A monomial ideal $J$ is called
 \emph{stable} if it holds
$$\tau \in J, \ x_j >\min(\tau) \Longrightarrow \frac{x_j\tau}{\min(\tau)}\in J$$

A monomial ideal $J$ is called \emph{quasi stable} if it holds
$$\tau \in J, \ x_j >\min(\tau) \Longrightarrow \exists t \geq 0  : \ \frac{x_j^t \tau}{\min(\tau)}\in J.$$
\end{definition}
We will show that this notion of quasi stable ideal coincides with the one in \cite{S1}, by proving that $J$ actually has a Pommaret basis.
\begin{remark}\label{StrongComplF}

\begin{itemize}
\item Obviously, a stably complete system  $M$  is also  stable, and a stable set is also quasi stable.

 \item In order to verify whether the conditions
above are satisfied for a given ideal $J$ it is
sufficient to check the terms in the basis $\cG  (J)$.
\end{itemize}
\end{remark}

\begin{proposition}\label{QstabFugualG}
Let  $J$ be a monomial ideal.
Then TFAE:
\begin{itemize}
 \item[i)]  $J$   is stable
 \item[ii)] $\mathcal{F}(J)=\cG  (J)$
\end{itemize}

\end{proposition}
\begin{proof}
{\emph{ i) $\Rightarrow$ ii)}} The inclusion
$\cG  (J)\subseteq \mathcal{F}(J)$
 is true for every  monomial ideal by definition of star set.
We prove now that
 $\gamma\notin \mathcal{F}(J)$ for every term $\gamma\in J\setminus \cG  (J)$. \\
By hypothesis, $\exists \tau \in \cG  (J)$,
 such that $\gamma=\tau \sigma $ and $\sigma \neq 1$.\\
Let $x_k:=\min( \gamma)$.
 If $x_k \vert \sigma$,
then$\frac{\gamma}{\min(\gamma)}=\tau\frac{\sigma}{x_k}\in J$, so that $\gamma \notin
 \mathcal{F}(J)$.

 If, on the other hand, $x_k \not\vert \sigma$ and  $x_j$
is any variable dividing $\sigma$,
then $x_j>x_k$ and $x_k=\min(\tau) $.
By the
 stability of $J$ we have
 $\frac{x_j\tau}{x_k}\in J$, hence
$\frac{\gamma }{x_k}=\frac{\tau \sigma}{x_j}\frac{x_j}{x_k} \in J$,
 hence again $\gamma \notin \mathcal{F}(J)$.

 \medskip

 {\emph{ ii) $\Rightarrow$ i)}} 
 If \emph {ii)} holds,
 then   $\cG  (J)$ is the only stably complete system generating $J$.
  By remark \ref{StrongComplF},
 we can check the stability on the terms $x^\alpha \in \cG  (J)$.
Let $x_j> x_k:=min(x^\alpha)$.
  By hypothesis there exists $x^\beta \in \cG  (J)$ such
that $x_jx^\alpha \in \off_{\cG  (J)}(x^\beta)$, and, since $x^\alpha \in \cG  (J)$,  of
course $x^\alpha x_j\notin \cG  (J)$. Hence
  $x^\beta \vert \frac{x_jx^\alpha}{x_k} $ and so
 $\frac{x^\alpha x_j}{x_k} \in J$.
\end{proof}

\begin{proposition}\label{Pocostabile}
Let  $J$ be a monomial ideal.
Then TFAE:
\begin{itemize}
 \item[i)]  $J$   is  quasi stable
 \item[ii)] $\vert \mathcal{F}(J) \vert <\infty$
\item[iii)]    $\mathcal{F}(J)=\mathcal{H}(J)$ is the Pommaret basis of $J$.
\end{itemize}
\end{proposition}

\begin{proof}
{\emph{ i) $\Rightarrow$ ii)}} Let $a$ be the maximum
of the degrees of elements in $ \cG  (J)$ and let $t$ be such that
$\frac{x_j^t x^\alpha}{\min (x^\alpha)} \in J$  for every $x^\alpha \in  \cG  (J)$ and
$x_j>min(x^\alpha)$. We prove that $\mathcal{F}(J)$ is
contained in $P_{< d}$ where  $d:=a+tn$.
 Let $x^\alpha x^\eta \in J_{\geq d}$ with $x^\alpha \in \cG  (J)$ and  $x_k$ be $\min(x^{\alpha}x^\eta)$.
If $x_k\vert x^\eta$, then  obviously
 $\frac{ x^\alpha x^\eta}{x_k}= x^\alpha\frac{  x^\eta}{x_k} \in J$, so $x^{\alpha}x^\eta \notin \mathcal{F}(J)$. If, on the other hand,  $x_k\not\vert x^\eta$,
 then $x_k=\min(x^\alpha)$.
 Moreover, every variable dividing $x^\eta$ is higher
than $x_k$ and at least one of them , let us call it $x_j$,
 appears in $x^\eta$ with exponent $\geq t$,
 as $\deg(x^\eta)\geq nt$. Then $\frac{ x_j^t x^\alpha }{x_k}\in J$,
 hence $ \frac{ x^\alpha x^\eta }{x_k}=\frac{ x_j^t x^\alpha }{x_k}\cdot \frac{ x^\eta }{x_j^t} \in J$ and
 $x^\alpha x^\eta \notin \mathcal{F}(J)$.

 \medskip
 {\emph{ ii) $\Rightarrow$ iii)}} 
 By \emph{ii)} $ \mathcal{F}(J)$ is finite, and by 
 \ref{moltiplicative} is stably complete, so it is clearly the Pommaret basis of $J$.\\

  {\emph{ iii) $\Rightarrow$ i)}} By remark \ref{StrongComplF}, we check the quasi stability on the terms $x^\alpha \in \cG  (J)$.
Let $x_j> x_k:=min(x^\alpha)$.
  By the hypothesis on the finiteness of $\mathcal{F}(J)$, there exists $m\gg 0$ such that $x^\alpha x_j^m \notin \mathcal{F}(J)$. Moreover, being $\mathcal{F}(J)$
  a stably
 complete system, there exists
  $x^\beta \in \mathcal{F}(J)$ such that $x_j^m x^\alpha \in \off_{\mathcal{F}(J)} (x^\beta)$ and
  $x^\beta \vert \frac{x_j^m x^\alpha}{x_k} $.
Therefore,  $\frac{x_j^m x^\alpha}{x_k} \in J$,
namely $J$ is quasi stable.
\end{proof}

\begin{example}
In $k[x,y,z]$ with $x<y<z$:
\begin{itemize}
\item considered $J=(z,y^2) $,  we get $M=\mathcal{F}(J)=\cG  (J)=\{ z, y^2 \} $, since $J$ is  stable;
 \item taken the ideal $J'=(z^2,y)$, we get  $M=\mathcal{F}(J)= \{  z^2  , yz, y \}\supset \cG  (J) $.\\
 In fact, $J$ is quasi stable, but it is not  stable;
\item given $J=(y)$, the star set is $M=\mathcal{F}(J)=\{ z^k y \ \vert \ k\geq 0 \}$, and it holds $\vert \mathcal{F}(J) \vert = \infty$, since $J$ is not stable.
\end{itemize}
\end{example}

\section{ $M$-marked sets  and  reduction process.}\label{sec:reductiont}

In this section,
 we generalize
the notions of $J$-marked
 polynomial, $J$-marked
 basis and $J$-marked
family given in
\cite{BCLR, CR}
 for $J$ strongly
stable. \\
In those papers,
 the involved polynomials
 are marked on the monomial
 basis of the given monomial ideal $J$. 
Here, we give the analogous
 definitions for any
 monomial ideal, provided
 that the involved
polynomials are marked
on a complete generating
 system in the sense
of definition
\ref{Complete}.\\
After determining
the setting, we extend
to it the reduction
process of the
quoted papers.\\
At the end, we will
 see that such
a generalized
procedure does not need to be 
noetherian for every
complete system of terms.
We will need to add
 some hypotheses on
the given complete
system in order to
overcome this problem.
\\
We point out that, as in \cite{BCLR, CR}, we do not introduce any term-ordering and this represents an important difference w.r.t. Janet's papers.
\\
Moreover, we consider polynomials with coefficients in a ring, not necessarily in a field. 

\begin{definition}\label{MMB}
Let $M$ be a complete system of terms and $J$
be the ideal it generates.
\begin{itemize}
\item
A $M$-\emph{marked set}
is a finite set $\mathcal{G}$ of
homogeneous (monic) marked
polynomials $f_\alpha=x^\alpha-\sum c_{\alpha\gamma} x^\gamma$,
with  $\Ht(f_\alpha)=x^\alpha \in M$ and
   $\supp (f_{\alpha}-x^{\alpha})\subset \cN  (J)$,
 so that  $\vert \supp (f)\cap J \vert =1$.
\item  A $M$-\emph{marked basis}  $\mathcal{G}$
is a  $M$-marked set  such that   $\cN  (J)$
 is a basis of $P/(\mathcal{G})$ as $A$-module, i.e.
 $P=(\mathcal{G})\oplus \langle \cN  (J) \rangle$
 as an $A$-module.
\item  The   $M$-\emph{marked family}
$\Mf(M)$ is the set of all homogeneous
 ideals $I$  that are generated by a $M$-marked basis.
\end{itemize}
\end{definition}

\begin{remark}
 Observe that the above definition of marked family  $\Mf(M)$ is consistent with that given in the Introduction of $\Mf (J)$ for a monomial ideal $J$.
 Indeed, if $I\in \Mf(M)$, then $I\in \Mf(J)$ with $J=(M)$. On the other hand, for every given $J$  there are  complete systems $M$ that   generate it, 
 for instance  $M=\mathcal F(J)$ and $\Mf(J)=\Mf(M)$. In fact, if   $I\in \Mf(J)$, every  polynomial $h$ 
 can be uniquely  written as a sum $ f+g$ with $f\in I$ and $g \in  \langle \cN (J) \rangle $; especially for every $x^\alpha \in M$, we have
\begin{equation}\label{EqnF}
 x^\alpha=f_\alpha+ g_\alpha, \; f_\alpha \in I \textrm{ and } g_\alpha \in \langle \cN (J) \rangle.
\end{equation}
Then $I$  contains  
 the $M$-marked basis $$\mathcal G=\{ f_\alpha=x^\alpha -g_\alpha \ , \ x^\alpha \in \mathcal M\}.$$ Furthermore $\mathcal G$ is a $M$-marked basis since
 $(\mathcal G)\subseteq I$ and $P=(\mathcal G) + \langle \cN  (J) \rangle=I\oplus  \langle \cN  (J) \rangle$.
 
 The only difference between the two notations $\Mf(J)$ and $\Mf(M)$ with $M$ a complete system generating $J$, is that using  the second one we present every 
 ideal of the family by means of a special set of generators depending on $M$. Note that, by the definition itself of $\Mf(J)$,
 we can assert that for every ideal $I \in \Mf(J)$ the $M$-marked basis generating it is \emph{unique}. 
\end{remark}

We define now a reduction
procedure for terms and
polynomials, with respect
 to an homogeneous set
$\mathcal G$ of polynomials,
 marked on a complete
system of terms $M$.
\\
The usual reduction
process with respect to
 $\mathcal G$ consists
of substituting each
term $x^\alpha x^\eta$,
multiple of an head term $x^\alpha=\Ht(f_\alpha)$,
 with the polynomial
$(x^\alpha-f_\alpha)x^\eta=g_\alpha x^\eta$.
\\
We add an extra
condition to the
standard procedure,
namely that this
substitution can be
performed only in the
case
$x^\alpha x^\eta=x^\alpha \ast_M x^\eta$.

\begin{definition}\label{def: riduzione}
Let  $M$ be a complete system and   $\mathcal G$ a $M$-marked set.
 We will denote by   $ \xrightarrow{\mathcal G} $ the
transitive closure of the 
relation $h  \xrightarrow{\mathcal G}
   h-c f_{\alpha}x^{\eta} $,
  where $x^\alpha x^\eta=x^\alpha \ast_M x^\eta$ is a term
that appears in $h$ with a non-zero coefficient $c$.
  We will say that  $ \xrightarrow{\mathcal G} $ is
noetherian if the length $r$ of any sequence
  $h=h_0    \xrightarrow{\mathcal G} h_1 \xrightarrow{\mathcal G} \dots \xrightarrow{\mathcal G} h_r$  is bounded
by an integer number $m=m(h)$. This
  is equivalent to say that if we continue
rewriting terms in this way we always
obtain, after a finite number of reductions,
a polynomial whose support
   is contained in $ \cN  (J)$.

  We will write $h  \xrightarrow{\mathcal G}_\ast
  g $ if  $h  \xrightarrow{\mathcal G} g $ and $\supp (g)\subset  \cN  (J)$.
\end{definition}

In general, the relation $\xrightarrow {\mathcal G}$ is
not noetherian, namely there are
 sequences of reduction  of infinite length.

\begin{example}
Let $M:=\{ xz,yz,y^2\}$ a set of terms in $k[x,y,z]$ with $x<y<z$.
We find the following sets of multiplicative variables:
\begin{itemize}
\item $\mult_M(xz)=\{x,z\}$
\item $\mult_M(y^2)=\{x,y \}$
\item $\mult_M(yz)=\{x,y,z\}$
\end{itemize}
and check that $M$ is complete. \\ Let $\mathcal G$ the $M$-marked set $\{f_{xz}=xz-xy, \ f_{yz}=yz-z^2,\ f_{y^2}=y^2\}$.\\
Then we have the infinite sequence of reductions:
$$xz^2=xz\ast_M z \xrightarrow{G} xz^2-f_{xz}z=xyz=yz\ast_M x \xrightarrow{\mathcal G} xyz-f_{yz}x= xz^2 $$
\end{example}

However, the reduction  $\xrightarrow {\mathcal G}$ is always noetherian if ${\mathcal G}$ is marked on a stably complete system. In order to prove
this fact we will use the following special subset of the ideal $({\mathcal G})$.

\begin{definition} Let $\mathcal G$ be a
$M$-marked set on a   complete system of
terms $M$ and let $J:=(M)$.
For each degree $s$,
 we will denote by $\mathcal G^{(s)}$
 the set of homogeneous
 polynomial
$$\mathcal G^{(s)}:=\{  f_\alpha x^\eta \ \vert \ x^\alpha \ast_{M} x^\eta \in (M)_s \}$$
marked on the terms of $J_s$ in the natural way
$\Ht ( f_\alpha x^\eta)=x^\alpha  x^\eta$.
\end{definition}

\begin{remark}\label{remark:minore} 
Observe that if
  $\mathcal G$ is a
$M$-marked set on a  stably complete system of terms $M$, for every homogeneous polynomial $g$ of degree $s$, $g \xrightarrow {\mathcal G} h$ implies that 
 $g-h= \sum _{i=1}^m c_i f_{\alpha_i} x^{\eta_i}\in \langle \mathcal G^{(s)}\rangle$.
 
It is worth noticing as a direct consequence
of  Lemma \ref{lem: minorelex}  that if $ f_\alpha x^\eta\in \mathcal G$, then
  every term in
$\supp( x^{\alpha} x^{\eta}- f_\alpha  x^\eta)$
either belongs to $\cN  ((M))$ or is of the type
 $x^{\alpha'} \ast_M x^{\eta'}$ with $x^{\eta'} <_{Lex} x^{\eta}$.
 \end{remark}

\begin{lemma}\label{lemma: pro ss} Let $\mathcal G$ be a
 $M$-marked set on the  stably  complete system of terms
$M=\mathcal F(J)$. \begin{enumerate}
 \item Every term in
$\supp( x^{\beta} x^{\epsilon}- f_\beta x^\epsilon )$
either belongs to $\cN  ((M))$ or is of the type
 $x^{\alpha} \ast_M x^{\eta}$ with $x^{\eta} <_{Lex} x^{\epsilon}$.
 \item  If $f_\beta \in \mathcal F (J)$,  then  all the polynomials $ f_{\alpha_i} x^{\eta_i}\in \mathcal G^{(s)} $
 used in the reduction  of $ x^\beta x^\epsilon$ (except $ f_\beta x^\epsilon$ if it belongs to $\mathcal G^{(s)}$)  are such that  $x^\epsilon  >_{Lex} x^{\eta_i}$.
  \item If $g= \sum _{i=1}^m c_i f_{\alpha_i} x^{\eta_i}$,
with $c_i\in k-\{0 \} $ and $ f_{\alpha_i} x^{\eta_i}\in \mathcal G^{(s)} $
pairwise different,
then $g\neq 0$ and  its support contains
 some term of the ideal  $J$.
                   \end{enumerate}
\end{lemma}
\begin{proof}
 (1) is a direct consequence of Lemma \ref{lem: minorelex}.
 
 (2) Assume that the statement holds for every  term $x^{\beta'} x^{\epsilon'}$, with $x^{\epsilon'} <_{Lex} x^{\epsilon}$.
 At a first step of reduction of $x^{\beta}x^{\epsilon}$ we use the polynomial
 $ f_{\alpha}x^{\eta}$ where $ x^\beta x^\epsilon=x^{\alpha} *_M x^{\eta}$, so that
 $x^{\eta}\leq _{Lex}x^\epsilon$;  moreover every term in the support of the obtained polynomial either belongs to $\cN  ((M))$ or is of the type
 $x^{\alpha'} \ast_M x^{\eta'}$ with $x^{\eta'} <_{Lex} x^{\eta}$ (Remark \ref{remark:minore}).  Then we conclude since we assumed the property 
 holds for all those terms.

(3)  
 We assume that the summands in $g$ are ordered
so that $x^{\eta_1} \geq_{Lex} x^{\eta_i}$ for
every $i=1, \dots, m$ and
 show that $x^{\eta_1+\alpha_1}$ belongs to
 the support of $g$.

 The term $x^{\alpha_1+\eta_1}$  cannot appear
 as the head of   $f_{\alpha_i}x^{\eta_i} $ for
some $i\neq 1$ because the star decomposition of
a term is unique.
 Moreover it cannot appear in $f_{\alpha_i}x^{\eta_i} -x^{\alpha_i+\eta_i} $
 since $x^{\alpha_1 + \eta_1}=x^\beta x^{\eta_i}$, with $x^\beta \in \cN(J)$
would imply
 $x^{\eta_i}>_{Lex}x^{\eta_1}$ (see Lemma \ref{lem: minorelex}),
 against the assumption.
\end{proof}

\begin{theorem} \label{RiduSC}

Let $\mathcal G$ be a $M$-marked set
on a stably complete system of terms $M$ and let
$J$ be the ideal generated by $M$.

Then the reduction process $\xrightarrow{\mathcal G}$ is
 noetherian and, for every integer $s$,
  $P_s=\langle  \mathcal G^{(s)} \rangle \oplus \langle \cN  (J)_s \rangle$.
Indeed, for every $h\in P_s$
$$h=f+g \hbox{  with } f\in \langle  \mathcal G^{(s)} \rangle \hbox{  and  }
 g\in  \langle \cN  (J)_s \rangle  \Longleftrightarrow  h \xrightarrow{\mathcal G}_\ast g \hbox{ and }
   f=h-g $$
\end{theorem}

\begin{proof}
Let
$\mathcal G=\{f_\alpha \ \vert \ x^\alpha \in M \}$.
\\
We observe that we have
%a direct sum
$\langle  \mathcal G^{(s)} \rangle \cap  \langle \cN  (J)_s \rangle =\{0\}$
 by Lemma \ref{lemma: pro ss}.

In order to prove that
the module $\langle  \mathcal G^{(s)} \rangle +  \langle \cN  (J)_s \rangle$ coincides with $P_s$ it
is sufficient to show that
it contains all the  terms in $J_s\setminus M$,
being obvious for those in $M$, for which $x^\alpha =f_\alpha+g_\alpha$ (see \ref{EqnF}).
%the other ones.

Let $\tau $ be a term in $J_s$.\\
If $\tau=x^\alpha\ast_M x^\eta$, we may assume
of having already proved the statement for  all the terms
$\tau'=x^{\alpha'} \ast_M x^{\eta'}$
with $x^{\eta'} <_{Lex} x^\eta$.

We have $x^\alpha x^\eta=f_\alpha x^\eta  +(x^\alpha - f_\alpha)x^\eta$ where $\supp (x^\alpha - f_\alpha)\subset \cN  (J)$.
If $x^\beta$ is any term
in this support, then either
$x^{\beta+\eta}\in \cN  (J)$ or $x^{\beta+\eta}=x^{\alpha'} \ast_M x^{\eta'}$ with $x^{\eta'} <_{Lex} x^\eta$ by
Lemma \ref{lem: minorelex}. This
allows us to conclude $P_s=\langle  \mathcal G^{(s)} \rangle +  \langle \cN  (J)_s \rangle$.

\medskip

Finally, in order to prove
 that $\xrightarrow{\mathcal G}$ is noetherian it is sufficient to observe that every step of reduction substitutes a term
of $J$ of the type $x^{\alpha} \ast_M x^{\eta}$ with $x^{\alpha}x^{\eta}- f_\alpha x^\eta$.
Indeed, by remark \ref{remark:minore}, each $\tau \in \supp(x^{\alpha} x^{\eta} -  f_\alpha x^\eta) \setminus \cN  ((M))$ has the form  $x^{\alpha'} \ast_M x^{\eta'}$, $x^{\eta'} <_{Lex} x^{\eta}$ and this permits to conclude by induction.
\end{proof}

As a straightforward consequence of the previous result, we obtain the following
\begin{corollary}\label{cor:equiv} If $M$ is a stably complete system and $\mathcal G$ is a $M$-marked set, the following are equivalent:
\begin{itemize}
\item $\mathcal G$ is a $M$-marked basis
\item for every $s$:  $\langle  \mathcal G^{(s)} \rangle  =(\mathcal G)_s$ 
\item for every $h\in (\mathcal G)$: $h \xrightarrow{\mathcal G}_\ast 0$
\item if $h-g\in (\mathcal G)$ and  $\supp(g)\subset  \cN(J)$, then 
$h \xrightarrow{\mathcal G}_\ast g$.
\end{itemize}    
\end{corollary}

\begin{remark}
We point out that if $\mathcal G$ is a $M$-marked set, but not a $M$-marked basis, then there are polynomials in the ideal $(\mathcal G)$ whose support is contained in $\cN ((M))$. Hence, we do not have a "normal form" of a   polynomial $h$ modulo $(\mathcal G)$, since, in general, there are several polynomials $g'$ such
 that $\supp(g')\subset  \cN(J)$
and $h-g'\in (\mathcal G )$. 
 However, the
reduction process $h \xrightarrow{\mathcal G}_\ast g$
with respect to a ${\mathcal F}  (J)$-marked set $\mathcal G$ gives a
unique reduced polynomial $g$ for every polynomial $h$.
\end{remark}

 Using the reduction process, we can now answer Problem \ref{Classif} and characterize
the ideals $I$ that belong to the marked family $\Mf (J)$. 

\begin{theorem}\label{ Criterio}
Let $\mathcal{G}$  be a $\mathcal{F}(J)$-marked set. Then:
\[  (\mathcal{G}) \in \Mf (J) \Longleftrightarrow \forall   f_{\beta}\in {\mathcal G}, \ \forall x_i >\min (x^\beta): \ f_{\beta}x_i \xrightarrow{\mathcal G}_\ast 0
\] 
\end{theorem}
\begin{proof}
Since "$\Rightarrow$" is a straightforward consequence of Corollary \ref{cor:equiv},  we only prove "$\Leftarrow$". More precisely, 
we prove that $(\mathcal{G})_m=(\mathcal{G}^{(m)})$, showing that if $f_{\beta}  \in \mathcal{G}$ and   $\deg(x^{\beta+\epsilon})=m$, then $f_{\beta} x^\epsilon  $ is
either an element of $\mathcal{G}^{(m)}$ itself or a linear combination of polynomials in $\mathcal{G}^{(m)}$.
\\
If this were not true, we can choose an element   $ f_{\beta}x^\epsilon \notin
\langle \mathcal{G}^{(m)}\rangle $ with $x^\epsilon  $ minimal with respect to
$<_{Lex}$.
As  $f_{\beta}x^\epsilon   \notin  \mathcal{G}^{(m)} $, at
least one variable $x_i$  appearing in $x^\epsilon $ with nonzero exponent is
non-multiplicative for $x^\beta$.
Let    $x^\epsilon   =x_i x^{\epsilon'}$.  By hypothesis
$f_{\beta}x_i \xrightarrow{\mathcal G}_\ast 0$, so that  $f_{\beta}x_i$  is a
linear combination $\sum c_i  f_{\alpha_i}x^{\eta_i}$ of polynomials
in $\mathcal{G}^{(\vert \beta \vert +1)}$. By Lemma \ref{lemma: pro ss} we have
$x^{\eta_i}<_{Lex} x_i$.
\\
Now $f_{\beta}x^\epsilon= (  f_{\beta}x_i)x^{\epsilon'}=(\sum c_i
f_{\alpha_i}x^{\eta_i})x^{\epsilon '}=\sum c_i  f_{\alpha_i}x^{\eta_i+\epsilon
'}$,
where $x^{\eta_i+\epsilon '}<_{Lex} x_i x^{\epsilon '}=x^\epsilon $.
Now we get a contradiction, since $ f_{\alpha_i}x^{\eta_i+\epsilon'} \in \langle
\mathcal{G}^{(m)}\rangle $ by the minimality of $x^\epsilon $.
\end{proof}

\begin{example}
Let $J$ be the monomial ideal $(x^3, xy,y^3)$ in $k[x,y]$ with $x<y$. Its star set is $\mathcal F(J)=\{ x^3,xy,x y^2,y^3\}$.  Using the criterion given in Theorem \ref{ Criterio}, we can easily check  that the $\mathcal F (J)$-marked set $\mathcal G:=\{f_1:={\bf x^3 }, \ f_2:={\bf xy}-x^2-y^2,\ f_3:={\bf xy^2},\ f_4={\bf y^3} \} $ (in bold the head terms) is a  $\mathcal F (J)$-market basis:
\begin{itemize}
\item $yf_1=xf_1+x^2f_2+xf_3\xrightarrow{\mathcal G}_\ast 0$, 
\item $yf_2=f_1-xf_2-f_4\xrightarrow{\mathcal G}_\ast 0 $
\item $yf_3=xf_4 \xrightarrow{\mathcal G}_\ast 0$.\end{itemize} 
This is a simple example of a marked basis which is not a \Gr\ basis. In fact, 
 it is obvious that $\Ht(f_2)=xy$ cannot be the leading term of $f_2$ with respect to any term-ordering  and, more generally, that  $J$ cannot be the initial ideal of the ideal $(\mathcal G)$, even though $(\mathcal G) \oplus \cN(J)=k[x,y]$.

A wider family of ideals of this type are presented in \cite[Example 3.18 and Appendix]{CR}. 
\end{example}

\begin{remark}\label{Seiler}
Observe that  we can perform the first step of reduction  of the polynomial $f_{\beta}x_i$  rewriting the head $x^\beta x_i $ throughout
$ f_{\alpha} x^\eta$ with $x^\beta x_i= x^\alpha x^\eta \in \off(x^\alpha)$. In this way we obtain $f_{\beta}x_i\xrightarrow{\mathcal G}
f_{\beta}x_i- f_{\alpha}x^\eta $,
namely the \emph{$S$-polynomial} $$S(f_{\beta}, f_{\alpha}):=\frac{lcm(x^\beta,x^\alpha)}{x^\beta}f_\beta- \frac{lcm(x^\beta,x^\alpha)}{x^\alpha}f_\alpha.$$
Therefore we could reformulate the criterion given by  Theorem \ref{ Criterio} as follows:
\[( \mathcal{G})\in \Mf (J) \Longleftrightarrow \forall  f_\alpha,  f_{\beta}\in {\mathcal G}: \  S(f_\alpha , f_{\beta}) \xrightarrow{\mathcal G}_\ast 0. 
\] 
However Theorem \ref{ Criterio} shows that it is sufficient to check a special subset of the $S$-polynomials that corresponds to the   basis for
the first syzygies of the terms in $\mathcal F(J)$. If $J$ is quasi stable, this basis is the one considered in   \cite{S0}.   It is obvious that 
the maximal degree of these special $S$-polynomials cannot exceed $1+\max\{\deg(x^\alpha) \ \vert \ x^\alpha \in \mathcal F (J)\}$.  
Indeed, if $J$ is quasi stable, $reg(J)=\max\{deg(\tau),\, \tau \in \mathcal{F}(J) \}$
as proved in \cite{J1,S1,S2}. 
\end{remark}

\begin{remark}\label{VWridu}
If $J$ is a quasi stable monomial ideal and  
$\mathcal G$ is a $\mathcal F(J)$-marked set,
then 
 there are only a finite number of reduction  to perform in order to decide if a $\mathcal F(J)$-marked set  $\mathcal G$ is a basis. 
 We will use this algorithm in order to endow the marked family  $\Mf(J)$ of a structure of affine scheme
 
 If the considered monomial ideal is not quasi stable, then the (unique) stably complete generating set is \emph{infinite}. 
Actually this does not necessarily exclude we can exploit it even from a computational point of view.
\end{remark}

\section{Marked families,   schemes and functors}\label{MFam}

In this section we follow  \cite{BCLR,CR} and  show how it is possible to  associate a scheme to each marked family  $\Mf (J)$. 
 Due to the naturality of this construction, we can mimic that  of     \cite{LR}, 
and define marked families as functors.

Our results are very similar, but more general, than those of \cite{BCLR,CR,LR}; in fact in those papers the ideal $J$ is assumed to be strongly stable.
Recall that 
a monomial ideal $J$ is called \emph{strongly stable}
if for every term $\tau\in J$ and pair of variables
$x_i,\ x_j$ such that $x_i\vert \tau$ and $x_i<x_j$,
then also $
\frac{\tau x_j}{x_i} $ belongs to $J$.

Obviously, a strongly stable ideal is also stable, so that $\mathcal F (J)=\cG (J)$.   
If $J$ is strongly stable, the notions of  
$\cG  (J)$-marked sets,  $\cG  (J)$-marked bases and $\cG  (J)$-marked family introduced in the previous sections  
exactly correspond to those of  $J$-marked sets,  $J$-marked bases, $J$-marked family
 considered in   \cite{BCLR,CR}  and the reduction procedure
$\xrightarrow{\mathcal G}$ with respect to a $\cG  (J)$-marked
set ${\mathcal G}$ introduced in definition
 \ref{def: riduzione} coincides
with the one used in those papers.

Moreover, for such an ideal $J$, the   scheme structure that we will define   is the same obtained in 
\cite{BCLR,CR} and used in \cite{BLR, LR} for a local study of Hilbert schemes.
Indeed,  for every monomial ideal $J$, if $I\in \Mf(J)$, then  
 the ideals $I$ and $J$ share the same Hilbert polynomial (and also the same Hilbert function), so that they correspond to points in the same Hilbert scheme.

The scheme   we associate to $\Mf(J)$ only depends on the monomial ideal $J$, 
but the way we use in order to define it  needs a set of generators $M$  complete,  finite and such that for every $M$-marked set $\mathcal G$ the  
reduction procedure 
$\xrightarrow{\mathcal G}$ is noetherian. 

\medskip

{\bf{ Then, in the following $J$ will be a  quasi stable monomial ideal and   $M$ will be its finite  star-set $\mathcal F(J)$, 
that is its Pommaret basis  $\mathcal H(J)$. }}

\medskip

Let $\{x^\alpha_1,...,x^\alpha_s\}$ be the monomials in $M$ and consider the polynomial  ring $B:=A[C]$, where $C$ is a compact notation for the set of variables
$C_{i,\beta}$  $i=1, \dots, s$ and $x^\beta \in \cN(J)_{\vert \alpha_i\vert}$. 
We also define the $M$-marked set in $B[x_1,...,x_n]$
$$\mathcal{G}:=\{f_{\alpha_i}:=x^{\alpha_i}+\sum C_{i,\beta}x^{\beta} \ \vert \, x^{\beta} \in \cN(J)_{{\vert \alpha_i\vert}} , \Ht (f_{\alpha_i})=x^{\alpha_i}\}. $$
Clearly, every $M$-marked set can be obtained specializing  $\mathcal G$, namely as $\phi (\mathcal G)$ for  a suitable morphism of $A$-algebras 
$\phi : A[C] \rightarrow A$. Moreover, by the uniqueness of the $M$-marked basis generating 
 each ideal in $\Mf(J)$, we can assert that for every ideal $I\in \Mf(J)$ there exists a unique specialization $\phi$ such that $(\phi(\mathcal G))=I$.
\\

We use Theorem \ref{ Criterio} in order to construct a set of polynomials $\mathcal R$ that will define the scheme we associate to $M$. 
If $g$ is a polynomial in $B[x_1,...,x_n]$, we  denote with  $\mathrm{coeff}_x(g)$ the set of coefficients of $g$ with respect to the only set of  variables 
$x_1, \dots ,x_n$; 
hence $\mathrm{coeff}_x(g)\subset B=A[C]$ is a set of polynomials in the variables $C$.
For every $x^{\alpha_i} \in M$ and  $x_j>\min (x^{\alpha_i})$, let $g_{\alpha_i,j}\in B[x_1, \dots, x_n]$ be such that
$f_{\alpha_i} x_j\xrightarrow{\mathcal G}_* g_{\alpha_i,j}$.

\begin{definition}\label{DefSchema}
Let $M$ be a stably complete system in $\mathcal{T}$, $A$ be any ring, and $\mathcal R$ be the union of $\mathrm{coeff}_x(g_{\alpha_i,j})$ for every $x^{\alpha_i} \in M$ and   $x_j>\min(x^{\alpha_i})$.

 We will call $M$-\emph{marked scheme} over the ring $A$, and  denote  with $\MFScheme{M}(A)$ the affine scheme $\mathrm{Spec} (A[C]/(\mathcal{R}))$.
\end{definition}

\begin{remark}
Every $M$-marked set in $A[x_1, \dots, x_n]$ is a $M$-marked basis if and only if  the coefficients of the terms in the tails satisfy the conditions given by $\mathcal R$. 

In particular, if $A=k$ is an algebraically closed field, then the closed points of $\MFScheme{M}(A)$   correspond to the  ideals in    the marked family $\Mf(J)$ where $J$ is the ideal in $k[x_1, \dots, x_n]$ generated by $M$. 
\end{remark}

\begin{remark}\label{Funtoriamo}
The above construction of  $\mathcal{R}$ is in fact \emph{independent} from  the fixed commutative  ring $A$, in the sense that it is preserved by extension of scalars.
We can first  choose $\ZZ$ as the coefficient ring and then  apply the   standard map $\ZZ \rightarrow A$.

  More formally,   for every stably complete set of terms $M$ we can define a functor
 between the category 
of $\ZZ$-algebras to the category of sets
\[ 
\MFFunctor{M}: \underline{\ZZ\text{-Alg}} \rightarrow \Sets
\]
that associates to any $\ZZ$-algebra $A$ the set 
$
\MFFunctor{M}(A):=\Mf(MA[x_1, \dots, x_n]) $
and to any morphism $\phi: A \rightarrow B$ the map
\[
\begin{split}
\MFFunctor{J}(\phi):\ \MFFunctor{M}(A)\ &\longrightarrow\ \MFFunctor{M}(B)\\
\parbox{1.5cm}{\centering I}\ & \longmapsto\ I \otimes_A B.
\end{split}
\]

Moreover, again following  \cite{LR},  it is possible to prove  that $\MFFunctor{M}$ is a representable functor represented by the scheme $\MFScheme{M}(\ZZ)= \mathrm{Spec}(\ZZ[C]/(\mathcal{R}))$.
\end{remark}

\section{Historical notes.}\label{Hist}
Through the trivial interepretation of  derivatives
$$\frac{1}{\alpha_1!\cdots \alpha_n!}
\frac{\partial^{\alpha_1+\alpha_2+\ldots+\alpha_n}}{\partial x_1^{\alpha_1}\partial
x_2^{\alpha_2}\ldots \partial x_n^{\alpha_n}},$$ in terms of  the corresponding term
$\tau = x_1^{\alpha_1}x_2^{\alpha_2}\ldots x_n^{\alpha_n}\in \mathcal{T}$,
Riquier \cite{Riq1,Riq2,Riq3} was able to
algebraically transform the problem of solving differential partial equations in terms of ideal membership.

After introducing the concept (but not the notion) of S-polynomials he proved that
 if the normal form (in terms of Gauss-Buchberger reduction) of each S-polynomial among the elements of the basis $\mathcal{G}$ goes to zero then
\begin{itemize}
\item the given basis $\mathcal{G}$ generates the related ideal;
\item the generic solution of the PDE can be given (and computed) as series in terms of initial conditions which can be described and formulated in terms of
 a Hironaka-Galligo-like decomposition \cite{Hir,Gal} (but more general) of the related {\em escalier} ${\sf N}$;
\end{itemize}
 if not all normal forms are 0, then, exactly as in Buchberger Algorithm, the non-zero normal forms are included in the basis and the procedure is repeated.

For instance, the system \cite[pp.188-9]{Riq3}
$$\frac{\partial^3 u}{\partial y^3} = A(x,y,z),\quad
\frac{\partial^2 u}{ \partial x\partial z} = B(x,y,z),\quad
\frac{\partial^3 u}{\partial x^2\partial y} = C(x,y,z),$$
must satisfy the integrability conditions
$$\frac{\partial^2 A}{ \partial x\partial z} = \frac{\partial^3 B}{\partial y^3},\quad
\frac{\partial^2 A}{\partial x^2} = \frac{\partial^2 C}{\partial y^2},\quad
\frac{\partial^2 B}{ \partial x\partial y} = \frac{\partial C}{\partial z};$$
in which case the initial conditions have the shape
$$\left\{\begin{array}{rcl}
%%%%%%%%%%%%%%%%%%%%%%%%%%%%%%%%%%%%%%%%%%%%%
\left.\begin{array}{rcl}
u &=& \phi_0(z)\cr
\frac{\partial u}{ \partial y} &=& \phi_1(z)\cr
\frac{\partial^2 u}{\partial y^2} &=& \phi_2(z)\cr
\end{array}\right.&
\left.\begin{array}{c}\cr\cr\cr\end{array}\right\}&
x-x_0=y-y_0=0,\cr
%%%%%%%%%%%%%%%%%%%%%%%%%%%%%%%%%%%%%%%%%%%%%
&&\cr
\left.\begin{array}{rcl}
\frac{\partial u}{\partial x} &=& \alpha_0\cr
\frac{\partial^2 u}{ \partial x\partial y} &=& \alpha_1\cr
\frac{\partial^3 u}{\partial x\partial y^2} &=& \alpha_2\cr
\end{array}\right.&
\left.\begin{array}{c}\cr\cr\cr\end{array}\right\}&
x-x_0=y-y_0=z-z_0=0,\cr
%%%%%%%%%%%%%%%%%%%%%%%%%%%%%%%%%%%%%%%%%%%%%
&&\cr
\left.\begin{array}{rcl}
\frac{\partial^2 u}{\partial x^2} &=& \psi(x)\cr
\end{array}\right.&
\left.\begin{array}{c}\cr\end{array}\right.&
y-y_0=z-z_0=0.\cr
%%%%%%%%%%%%%%%%%%%%%%%%%%%%%%%%%%%%%%%%%%%%%
\end{array}\right.$$

In his theory, Riquier was assuming that the set $\mathcal{ T}$ of the terms was ordered by a term-ordering; he was mainly using \cite[p.67]{Riq3}
the deglex ordering    induced by $x_1 > x_{2}
> \cdots > x_n,$ but he gave a large class of term-orderings to which his theory was applicable; actually (but he never stated that) his characterization is the classical  one of all term-orderings \cite{Erd,Rob}.
He was however forced to restrict himself to degree-compatible term-orderings in order to be granted convergency.

In his gaussian reduction, Riquier, as Buchberger, considered as head term of each "marked" polynomial its maximal term.

In his considerations on generic initial ideal, Delassus \cite{Del1}, followed by Robinson \cite{Rob1}  used  (deg)-rev-lex   induced by
$x_1 < x_{2}  \cdots < x_n$ and the minimal term as head term of each ``marked'' polynomial.

In order to "harmonize" the two notations, Janet in \cite{J1, J4}   applied
deglex  induced by $x_1 < x_{2} < \cdots < x_n$ and
chose the maximal term as head term, but
 expressed all terms as (!)  $x_n^{\alpha_n} x_{n-1}^{\alpha_{n-1}}\ldots x_1^{\alpha_1}$, while in \cite{J2} went back to use deglex  induced by $x_1 > x_{2}
> \cdots > x_n.$

What is worst, in \cite{J3} Janet not only applied deglex  induced by $x_1 < x_{2} , \cdots < x_n$ but presented all results within his notation; so, in his presentation of Delassus's result, the head term is again, {\em \`a la} Buchberher, the maximal one.

This is not helpful, as regards his reformulation of the previous results on  generic initial ideals and stability; thus while, for Robinson \cite{Rob1,Rob2} and Gunther \cite{Gun1,Gun2} a  generic initial ideal $ \epsilon(I)$ satisfies
$$\mu\in \epsilon(I) ,\, x_h \mid \mu,  i< h \then x_i\frac{\mu}{x_h}\in \epsilon( I),$$
according  \cite{J3} the formula is
$$\mu\in \epsilon( I) , x_h \mid \mu,  i> h \then x_i\frac{\mu}{x_h}\in \epsilon( I).$$

Under the suggestion of Hadamard \cite{Pom}, Janet dedicated his doctorial thesis   \cite{J1}
to  a reformulation of Riquier's results in terms of   Hilbert's results \cite{Hil}.

In particular, given a finite set of monomials $M$, he associates to each term $\tau\in M$, as functions of its relation with the other elements of $M$, a set of  variables which he labels {\em multiplicative}
(Definition~\ref{multiplicative}) and a subset of terms in $(M)$ which he called his {\em class} and which we labelled as its  {\em offspring} and considered $M$ {\em complete}
(Definition~\ref{Complete}) when the disjoint offsprings of $M$ cover $(M)$.

He then gave  \cite[p.80]{J1}  a {\em proc\'ed\'e r\'egulier pour
obtenir un syst\`eme complet base d'un module donn\'e} which {\em
ne pourra se prolonger ind\'efiniment}; it simply consisted to enlarge $M$ with the elements $xt\notin\cup_{\tau\in M} \off_M (\tau),t\in M$, $x$ non-multiplicative for $t$.

Janet can now formulate  \cite[p.75]{J4} Riquier's procedure; we can assume to have a finite basis $\mathcal{G}\subset \mathcal{P}$; denoting $M=\{\cT(f) :f\in \mathcal{G}\}$,
\begin{itemize}
\item we enlarge $M$ in order to made it complete and at the same time
\item we similarly enlarge $\mathcal{G}$, adding $xg$ to $\mathcal{G}$ when we add $x\cT(g)\notin\cup_{\tau\in M} \off_M (\tau)$;
\item we then perform Riquier's test, which, for a complete systems, consists in computing the normal form of each element $xg, g\in \mathcal{G}$, $x$  non-multiplicative for $\cT(g)$.
\end{itemize}

Janet  \cite[p.112-3]{J1} further remarks (in connection with Hilbert's syzygy theory) that the reduction-to-zero of all such elements give a basis $S$ of the syzygy module of $\mathcal{G}$. Actually he repeatedly applied the same procedure to $S$, thus computing a   resolution of $\mathcal{G}$
and anticipating Schreyer's Algorithm \cite{Sch}.

Next, in 1924, Janet \cite{J2}
moved his interest in extending the study to the homogeneous case,
 adapting his approach
on one side to the solution of partial differential equation given by
E. Cartan \cite{Car1,Car2,Car3} via his characters and test
and on the other side to the introduction by Delassus \cite{Del1} of the concept of generic initial ideal
and the precise description of it given by  Robinson \cite{Rob1,Rob2} and Gunther \cite{Gun1,Gun2};
he thus discussed the notion of {\em syst\`eme de forms (de m\^eme ordere) en involution}. The notion, as he explains, is independent from the variable chosen and allows to assign to the system a series of values $\sigma_i^{(p)}, 1\leq i \leq n, p\in{\Bbb N}$ which \cite[p.87]{J4}
{\em sont \'evidem\-ment invariables lors\-q'on fait un
changement lin\'eaire et homog\`ene des variables ind\'ependantes}
 which, under the assumption of generality,
 allow to describe the structure of  the  {\em generic escalier} of the considered ideal.
\\
The procedure, given a finite set $\mathcal{G}$ of forms, repeatedly produces  {\em \`a
la} Macaulay a linear basis
$\mathcal{B}_p$ of $(\mathcal{G})_p$ by performing linear algebra on the set $\{x_ig : g\in \mathcal{B}_{p-1},
1\leq i\leq n\}$; termination is granted when the formula (\ref{c55Eq1}) below
is satisfied.

Given a homogeneous ideal $I\subset k[x_1,x_2,\ldots,x_n]$, where  the variables are assumed to be generic, so that $ \cN  (I)$ is stable, Janet
 defined \cite[pp.30-2]{J2},\cite[p.30]{J3},\cite[pp.90-1]{J4},\cite[p.93, p.99]{Pom}
multiplicative variables  according \ref{Molt2},
 introduced values $\sigma_i^{(p)}(I)$  (or $\sigma_i^{(p)}$ for short when no confusion is possible) for every   $1\leq i \leq n$, and $ p\in{\Bbb N}$, 
which can be described as
$$\sigma^{(p)}_i := \#\left\{\tau\in\cN  (I),\deg(\tau)=p, \min(\tau)=i\right\}
$$
and,  fixing a value $p$ and denoting
$\sigma_i := \sigma^{(p)}_i$, and $\sigma'_i := \sigma^{(p+1)}_i$
proved
\begin{proposition}[Janet]\label{c55P3} It holds,
\begin{enumerate}
\item $\sigma'_1 + \sigma'_2 + \ldots + \sigma'_n \leq
\sigma_1 + 2 \sigma_2 + \ldots + n \sigma_n$;
\item $\sum_{i=1}^n {\sigma'}_i = \sum_{i=1}^n i \sigma_i
\then {\sigma'}_j = \sum_{i=j}^n  \sigma_i$ for each $j$.
\item $\sum_{i=1}^n {\sigma'}_i =  \sum_{i=1}^n i \sigma_i
\then \sum_{i=1}^n {\sigma}^{(P+1)}_i =  \sum_{i=1}^n i \sigma^{(P)}_i$
for each $P > p$.
\end{enumerate}
\end{proposition}

He can then state
\begin{definition}[Janet]\cite[pp.90-1]{J4}\label{involutive}
A finite set $E\subset \mathcal{ P}$ of
forms of degree at most $p$
generating the ideal $I\subset P$,
is
said to be {\em involutive}\footnote{\em en involution.}\index{involutive} if, with the present notation,
it satisfes the formula
\begin{eqnarray}\label{c55Eq1}
\sum_{i=1}^n \sigma^{(p+1)}_i = \sum_{i=1}^n i \sigma^{(p)}_i.
\end{eqnarray}
\qed\end{definition}

Thus, once the iterated Macaualy-like procedure satisfies (\ref{c55Eq1}) at degree $\bar{p}$ then it successfuly terminates and  the {\em finite} bases produced by it is involutive; Janet  is therefore able to present the ideal
$\{\tau\in{\sf T}(I), \deg(\tau)\geq \bar{p}\}$ by explicitly producing\cite{J4} the decomposition
$$\{\tau\in{\sf T}(I), \deg(\tau)\geq \bar{p}\} = \sqcup_{\tau\in M} \off_M(\tau)$$
where $M$ is the stably complete set $M=\{\tau\in{\sf T}(I), \deg(\tau)\geq \bar{p}\}$
and to express its Hilbert polynomial as
$${^h}H_I(t) =
\sum_{h=1}^{n-1} \binom{t - p + h - 1}{h-1}
\sigma^{(p)}_h(I).$$
%%%%%%%%%%%%%%%%%%%%%%%%%%%%%%%
In our context, the characterization of $\sigma_i^{(p)}$ and definition \ref{involutive} lead to the following
\begin{proposition}\label{JQstabBase}
With the previous notation,
if $J$ is a quasi stable monomial
ideal, then
$$\sum_{i=1}^n
 \sigma_i^{(p+1)}(J)=
\sum_{i=1}^ni
\sigma_i^{(p)}(J).$$

The same equality  holds if $I$ is a homogeneous ideal generated by a $J$-marked basis $\mathcal G$ with $J$ quasi stable.

Therefore $\mathcal G$ is an  involutive basis.
\end{proposition}
\begin{proof}
For the first statement  we observe that if  $p\geq \overline{p}$ every term  $\tau \in J_{p+1}$  can be written in a unique way as  a product $\tau=\theta x_i$, with  $\theta\in  J_{p}$ and
 $x_i$ a multiplicative variable for $\theta$, i.e.  $x_i \leq \min(\theta)$.
 
 If $I$ is the homogeneous ideal generated by a $J$-marked set $\mathcal G$, then 
  for the corresponding $f_\tau \in \mathcal{G}^{(p+1)}$ we have  $f_\tau=f_\theta x_i$ with $f_\theta$ in  $\mathcal{G}^{(p)}$ and of course $x_i \leq \min(\theta)$. 
  
 If  $\mathcal G$ is a 
$J$-marked basis, then we get the equality since  $(\mathcal{G})_{t}=(\mathcal{G}^{(t)})$ for every $t$ (Corollary \ref{cor:equiv}).
\end{proof}
%%%%%%%%%%%%%%%%%%%%%%%%%%%%%%%

Note that for an ideal $I$ generated by a  $J$-marked set $\mathcal G$ which is not  a marked basis,
only  the inequality $\sum_{i=1}^n
 \sigma_i^{(p+1)}\leq
\sum_{i=1}^ni
\sigma_i^{(p)}$ holds true, since $(\mathcal{G})_{t}\supseteq (\mathcal{G}^{(t)})$.

\medskip

The iterated Macaualy-like procedure gives also a fine decomposition of
${\sf N}(I)_{ \geq \bar{p}-1}$ as follows:
\begin{itemize}
\item he partitions the set
${\sf N}(I)_{\overline{p}-1}$ as ${\sf N}_{\overline{p}-1} = \sqcup_{i=0}^{n-1} N_i$
associating to
\begin{itemize}
\item $N_0$ the monomials $\tau\in{\sf N}_{\overline{p}-1}(I)$ for which
$x_1\tau\in{\sf T}(I)$;
\item while each of the $\sigma_1$ elements
$\tau=\frac{\upsilon}{x_1}\in {\sf N}(I)_{\overline{p}-1}\setminus N_0,  \clss(\upsilon)  = 1$, is inserted in $N_i$ if it is one of the $\sigma_i$ elements which can be expressed as
$\tau=\frac{\upsilon_i}{x_i},  \clss(\upsilon_i) = i$
but is not one of the
$\sigma_{i+1}$ elements which can be expressed as
$\tau=\frac{\upsilon_{i+1}}{x_{i+1}},  \clss(\upsilon_{i+1}) = i+1$.
\end{itemize}
\item he then associate to each $\tau\in N_i$
$\mult(\tau)=\{x_j, 1\leq j\leq i\}$ as multiplicative variables and
$\off (\tau):=\{\tau \omega, \omega\in \mathcal{ T}[1,i]\}$ as its offspring
\item and states
$$\{\tau\in{\sf N}(I), \deg(\tau)\geq \bar{p}-1\} = \sqcup_{i=0}^{n-1}
\sqcup_{\tau\in N_i} \off (\tau).$$
\end{itemize}

Riquier's and Janet's results were introduced to the Computational Algebra commutative at the MEGA-90
Symposium in 1990 by a survey by Pommaret \cite{PomAk}  of his theory and, two years later, through a paper by F.
Schwarz \cite{Schw}
where he  remarked:
\begin{quote}
The concept of a Gr\"obner base and algorithmic methods for constructing it for a given system of
multivariate polynomials has been established as an extremaly important tool in commutative
algebra. It seems to be less well known that similar ideas have been applied for investigating
partial differential equations (pde's) around the turn of the century in the pioneering work of the
French mathematicians Riquier and Janet. [...] [T]heir theory [...] is basically a
critical-pair/completion procedure. All basic concepts like a term-ordering, reductions and
formation of critical pairs are already there.
\end{quote}

This prompted V. Gerdt to suggest his coworkers Zharkov and Blinkov to investigate whether the
results by Janet and Pommaret were translatable from pde's to polynomial rings in order to produce an
effective alternative approach to Buchberger's Algorithm; the conclusion of this investigation \cite{ZB,Z} was successful --- the proposed algorithm was able to give a
solution with a speed-up of 20 w.r.t. degrevlex Buchberger's algorithm on classical test-suites and caused sensation in the community.

Unfortunately, among the two constructions proposed by Janet, they hitted the involutive one, which is {\em not} a Buchberger-like procedure and presented it as such, remarking that in general does not terminate and that the basis is not necessarily finite  unless the ideal is 0-dimensional. What is worst, they attributed to Pommaret their mistakes, thus introducing in literature a ``bad'' fictional Pommaret division compared with the ``good'' Janet division (related to Janet completion \cite{J1} procedure).

An algorithm based on Janet's notion \cite{J1} of completeness is reported in \cite{BG1,BG2,G}

Involutiveness is the argument of the {\em Habilitation} thesis (2002) of Seiler \cite{SH,S0,S1}; an improved version has recently appeared as \cite{SB}. Finiteness is a required condition for the notion of Pommaret bases \cite{S2}.

\end{document}